\documentclass[12pt,letterpaper]{amsart}
\usepackage{palatino, euler, epic, eepic, amssymb, xypic, floatflt, microtype, enumerate, stmaryrd}
\input xy
\xyoption{all}

 \newlength{\baseunit}               
 \newcount{\numlines}                
\usepackage{amsmath, amssymb, amscd, verbatim, xspace,amsthm}
\usepackage{latexsym, epsfig, color}

\newcommand{\ra}{\rightarrow}

\newcommand\Sym{\operatorname{Sym}}

\newcommand\Spec{\operatorname{Spec}}
\newcommand\Proj{\operatorname{Proj}}

\newcommand\bq{\begin{equation}}
\newcommand\eq{\end{equation}}

\newtheorem{proposition}{Proposition}[section]
\newtheorem{theorem}[proposition]{Theorem}

\theoremstyle{definition}
\newtheorem{definition}[proposition]{Definition}

\theoremstyle{remark}
\newtheorem{remark}[proposition]{Remark}
\usepackage{url}

\numberwithin{equation}{section}

\newcommand{\cut}[1]{}

\newcommand\hidden[1]{}

\newcommand{\cC}{{\mathcal{C}}}

\newcommand{\cM}{\mathcal{M}}

\newcommand{\PP}{\mathbb{P}}

\newcommand{\Mp}{{\mathcal{M}^+}}                                     %
\newcommand{\Rel}{{\mathcal{R}}}                                      %
\newcommand{\ZZ}{{\mathbb{Z}}}                                        %
\newcommand{\LL}{{\mathbb{L}}}                                        %
\newcommand{\cO}{{\mathcal O}}                                        %
\newcommand{\ch}{\operatorname{\tilde{c}_1}}                          %
\newcommand{\Schk}{\operatorname{Sch}_{k}}                                         %
\newcommand{\Ab}{\operatorname{Ab}_{*}}                                            %
\newcommand{\id}{\operatorname{Id}}                                                %
\newcommand{\Kth}{G_{0}[\beta,\beta^{-1}]}                                         %
\newcommand{\fxy}{f:X \ra Y}                                                       %
\newcommand{\fyx}{f:Y \ra X}                                                       %
\newcommand{\fxz}{f:X \ra Z}                                                       %
\newcommand{\gxy}{g:X \ra Y}                                                       %
\newcommand{\gyz}{g:Y \ra Z}                                                       %
\title{Universality of K-Theory}

\author{Jos\'e Luis Gonz\'alez and Kalle Karu}
\address{J.L. Gonz\'alez,  Department of Mathematics, University of British Columbia,
  Vancouver, BC V6T1Z2, CANADA  \newline \indent
K. Karu,
Department of Mathematics, University of British Columbia, 
  Vancouver, BC V6T1Z2, CANADA} 
\email{jgonza@math.ubc.ca, karu@math.ubc.ca}
\thanks{This research was funded by NSERC Discovery and Accelerator grants.}

\begin{document}
\begin{abstract}
We prove that graded $K$-theory is universal among oriented Borel-Moore homology theories with a multiplicative periodic formal group law.
%
%
%
%
%
\end{abstract}
\maketitle
\setcounter{tocdepth}{1} 
\tableofcontents



\section{Introduction}


In \cite{DaiThesis} and \cite{DaiPaper} Shouxin Dai established a useful universal property enjoyed by the Gro\-then\-dieck $K$-theory functor of coherent sheaves $G_{0}$, that was conjectured by Levine and Morel in \cite{Levine-Morel}. Using results of Levine and Morel in \cite{Levine-Morel}, one can interpret Dai's theorem as describing some sufficient conditions under which $K$-theory can be recovered via extension of scalars from algebraic cobordism.
The argument that Dai provides in \cite{DaiThesis} and \cite{DaiPaper} yields the universality of $K$-theory for schemes that admit embeddings into smooth schemes in the category $\Schk$ of finite type separated schemes over a fixed field $k$ of characteristic zero.
We will build on Dai's work to show that this universality result holds over all of $\Schk$, and more generally over any \emph{l.c.i. closed admissible} subcategory of $\Schk$ as defined in \S~2.

The proper statement of Dai's universality theorem involves the notion of \emph{oriented Borel-Moore homology theories} (see \S~\ref{section.definition.OHT}). 
One can regard each of these theories as an additive functor with values on graded abelian groups, endowed with projective push-forwards, l.c.i. pull-backs and exterior products, which satisfy some natural axioms. 
%
%
Each oriented Borel-Moore homology theory has a notion of Chern classes, and it is proved in \cite{Levine-Morel} that for each such theory there is an associated formal group law that governs the relation between the first Chern class of the tensor product of two line bundles with the first Chern classes of each of the factors. 

The simplest example of an oriented Borel-Moore homology theory is the Chow homology functor $A_{*}$ described in \cite{FultonIT}, which has an additive formal group law. In fact, in characteristic zero, $A_{*}$ is universal among oriented Borel-Moore homology theories for which the first Chern class operators are additive with respect to tensor products of line bundles \cite[Theorem 7.1.4]{Levine-Morel}. On the other extreme, Levine and Morel constructed in \cite{Levine-Morel} an oriented Borel-Moore homology theory $\Omega_{*}$ called algebraic cobordism that is an algebraic analogue of the theory of complex cobordism and whose first Chern class operators behave on tensor products of line bundles according to the universal formal group law $(F_{\LL},\LL)$. Moreover, Levine and Morel have proved that in characteristic zero, algebraic cobordism is the universal oriented Borel-Moore homology theory on $\Schk$.
The Grothendieck $K$-theory functor of coherent sheaves $G_{0}$, when made into a graded theory by $\Kth$, provides another important example of an oriented Borel-Moore homology theory.  The formal group law of this theory is the universal multiplicative periodic formal group law (see \ref{sub.formalgl} and \ref{ex.K.theory}). 

Let $H_{*}$ be any oriented Borel-Moore homology theory with a multiplicative periodic formal group law defined on the full subcategory $\mathcal{S}$ of $\Schk$ consisting of the smooth quasiprojective schemes in $\Schk$. In \cite{Levine-Morel} Levine and Morel proved that there exists one and only one transformation $\Kth \ra H_{*}$ of oriented Borel-Moore homology theories on $\mathcal{S}$. 
In other words, they established that $\Kth$ is the universal multiplicative periodic oriented Borel-Moore homology theory on the category of smooth quasiprojective schemes over a field of characteristic zero. This led them to ask in \cite{Levine-Morel} whether this property still holds over all of $\Schk$.

Dai's universality result came close to answer Levine and Morel's question. One way to state Dai's result is that in characteristic zero, over any l.c.i. closed admissible subcategory $\mathcal{C}$ of $\Schk$ consisting of schemes that admit embeddings in smooth schemes in $\Schk$ (e.g. quasiprojective schemes), the natural morphism $\Omega_{*} \otimes_{\LL} \ZZ[\beta,\beta^{-1}] \ra G_{0}[\beta,\beta^{-1}]$ induced by the universality of algebraic cobordism is an isomorphism. In particular, it follows from the universality of algebraic cobordism that $K$-theory $G_{0}[\beta,\beta^{-1}]$ is universal among oriented Borel-Moore homology theories with a multiplicative periodic formal group law on any such category $\mathcal{C}$. 

The goal of this article is to extend Dai's universality to answer positively Levine and Morel's question. This is done in Theorem \ref{thm.universality} by using Dai's result along with a descent sequence for projective envelopes that holds both in algebraic cobordism and $K$-theory (see Theorems \ref{thm-Gillet} and \ref{thm-seq1}), obtaining that in characteristic zero the morphism $\Omega_{*} \otimes_{\LL} \ZZ[\beta,\beta^{-1}] \ra G_{0}[\beta,\beta^{-1}]$ is an isomorphism of oriented Borel-Moore homology theories over all of $\Schk$. In particular, in characteristic zero the graded $K$-theory functor $G_{0}[\beta,\beta^{-1}]$ is the universal multiplicative periodic oriented Borel-Moore homology theory on any l.c.i. closed admissible subcategory $\mathcal{C}$ of $\Schk$.

The descent sequence in $K$-theory that we need in the main argument was proved by Gillet in \cite{GilletSeq} as a corollary to a more general descent theorem for simplicial schemes. In the appendix we give a shorter proof of Gillet's theorem using Dai's universality result and the descent sequence for algebraic cobordism.

\subsection*{Acknowledgments}
This article builds on the result of Shouxin Dai establishing the desired universality property of $K$-theory for schemes that admit embeddings on smooth algebraic schemes.


\section{Oriented Borel-Moore Homology Theories}   \label{section.definition.OHT}

\subsubsection{} \label{setting}
Throughout this article all of our schemes will be defined over a fixed field $k$ of characteristic zero. We denote by $\Schk$ the category of separated finite type schemes over $\Spec k$. For any full subcategory $\cC$ of $\Schk$, we denote by $\cC'$ the subcategory of $\cC$ with the same objects but whose morphisms are the projective morphisms. Smooth morphisms are assumed to be quasiprojective, and by definition locally complete intersection (l.c.i.) morphisms are assumed to admit a global factorization as the composition of a regular embedding followed by a smooth morphism. We follow the convention that smooth morphisms, and more generally l.c.i. morphisms, have a relative dimension. If $\fxy$ is an l.c.i. morphism of relative dimension $d$ (or relative codimension $-d$) and $Y$ is irreducible, then $X$ is a scheme of pure dimension equal to $\operatorname{dim}Y+d$. 
$\Ab$ will denote the category of graded abelian groups.

\subsubsection{} \label{def.admissible}
We call a full subcategory $\mathcal{C}$ of $\Schk$ \emph{admissible} if it satisfies the following conditions
\begin{enumerate}
\item[1.] $\operatorname{Spec}k$ and the empty scheme are in $\mathcal{C}$.
\item[2.] If $Y \ra X$ is a smooth morphism in $\Schk$ with $X \in \mathcal{C}$, then $Y \in \mathcal{C}$. 
\item[3.] If $X$ and $Y$ are in $\mathcal{C}$, then so is their product $X \times Y$.
\item[4.] If $X$ and $Y$ are in $\mathcal{C}$, then so is their disjoint union $X \coprod Y$.
\end{enumerate}

\subsubsection{}  \label{def.closed.admissible}  
We call a full subcategory $\mathcal{C}$ of $\Schk$ \emph{l.c.i.-closed admissible} if it is admissible and it satisfies the following conditions 
\begin{enumerate}
\item[1.] If $Y \ra X$ is an l.c.i. morphism in $\Schk$ with $X \in \mathcal{C}$, then $Y \in \mathcal{C}$.   
\item[2.] If $i: Z \ra X$ is a regular embedding in $\cC$, then the blowup of $X$ along $Z$ is in $\cC$. 
\end{enumerate}

\subsubsection{} 
Let $\fxz$ and $\gyz$ be morphisms in an admissible subcategory $\cC$ of $\Schk$. We say that $f$ and $g$ are transverse in $\cC$ if they satisfy
\begin{enumerate}
\item[1.] $\operatorname{Tor}_{j}^{\mathcal{O}_{Z}}(\mathcal{O}_{Y},\mathcal{O}_{X})=0$ for all $j>0$.
\item[2.] The fiber product $X \times_{Z} Y$ is in $\cC$.
\end{enumerate}


\subsection{Oriented Borel-Moore Homology Theories}

\begin{definition}[Levine-Morel]  \label{def.OBMF}
Let $\mathcal{C}$ be an admissible subcategory of $\Schk$.
An \emph{oriented Borel-Moore homology theory} (or \emph{OBM homology theory}) $H_{*}$ on $\mathcal{C}$ is given by:

($\operatorname{D_{1}}$) An \emph{additive} functor $H_{*}: \mathcal{C}' \ra \Ab$, i.e., a functor $H_{*}: \mathcal{C}' \ra \Ab$ such that for any finite family $(X_{1},\ldots,X_{r})$ of schemes in $\mathcal{C}'$, the morphism
	\[
\bigoplus_{i=1}^{r}H_{*}(X_{i}) \ra H_{*}(\coprod_{i=1}^{r}X_{i})
\]
induced by the projective morphisms $X_{i} \subseteq \coprod_{i=1}^{r}X_{i}$ is an isomorphism.

($\operatorname{D_{2}}$) For each l.c.i. morphism $\fyx$ in $\mathcal{C}$ of relative dimension $d$ a homomorphism of graded groups
	\[
f^{*}:H_{*}(X)\ra H_{*+d}(Y).	
\]

($\operatorname{D_{3}}$) An element $1 \in H_{0}(\Spec k)$ and for each pair $(X,Y)$ of schemes in $\mathcal{C}$, a bilinear graded pairing
	\begin{align*}
\times: H_{*}(X) \times H_{*}(Y) &\ra H_{*}(X \times Y)   \\
(\alpha, \beta)	&\mapsto \alpha \times \beta,
\end{align*}
called the \emph{exterior product}, which is associative, commutative and admits $1$ as unit element.

These data satisfy the following axioms:

($\operatorname{BM_{1}}$) One has $\id_{X}^{*}=\id_{H_{*}(X)}$ for any $X \in \mathcal{C}$. Moreover, for any pair of composable l.c.i. morphisms $\fyx$ and $g:Z \ra Y$, of pure relative dimensions $d$ and $e$ respectively, one has 
	\[
(f \circ g)^{*} = g^{*} \circ f^{*}: H_{*}(X)\ra H_{*+d+e}(Z).	
\]

($\operatorname{BM_{2}}$) For any projective morphism $\fxz$ and any l.c.i. morphism $\gyz$ that are transverse in $\mathcal{C}$, if one forms the fiber diagram
\[
\xymatrix{
W \ar[d]^{f'} \ar[r]^{g'} &X \ar[d]^f \\
Y \ar[r]^g&Z }
\]
(note that in this case $f'$ is projective and $g'$ is l.c.i.) then
	\[
g^{*}f_{*}=f'_{*}g'^{*}.
\]

($\operatorname{BM_{3}}$) Given projective morphisms $f:X\ra X'$ and $g:Y\ra Y'$ one has that for any classes $\alpha \in H_{*}(X)$ and $\beta \in H_{*}(Y)$ 
	\[
	(f \times g)_{*}(\alpha \times \beta)   =    f_{*}(\alpha) \times g_{*}(\beta)   \in H_{*} (X' \times Y').
\]

Given l.c.i. morphisms $f:X\ra X'$ and $g:Y\ra Y'$ one has that for any classes $\alpha \in H_{*}(X')$ and $\beta \in H_{*}(Y')$
	\[
	(f \times g)^{*}(\alpha \times \beta)   =    f^{*}(\alpha) \times g^{*}(\beta)  \in H_{*} (X \times Y).
\]

($\operatorname{PB}$) For any line bundle $L \ra Y$ in $\mathcal{C}$ with zero section $s: Y \ra L$, define the operator $\ch(L): H_{*}(Y)\ra H_{*-1}(Y)$ by $\ch(L)=s^{*}s_{*}$. 
Let $\mathcal{E}$ be the sheaf of sections of a rank $n+1$ vector bundle $E$ on $X$ in $\mathcal{C}$, with associated projective bundle $q: \PP(\mathcal{E}) = \Proj_{\mathcal{O}_{X}}(\Sym_{\mathcal{O}_{X}}^{*}\mathcal{E}) \ra X$ and associated canonical quotient line bundle $q^{*}E \ra O_{\PP(\mathcal{E})}(1)$. 
For $i= 0,\ldots,n$, let
\[
\xi^{(i)}:H_{*+i-n}(X) \ra H_{*}(\PP(\mathcal{E})) 
\]
be the composition of $q^{*}: H_{*+i-n}(X) \ra H_{*+i}(\PP(\mathcal{E}))$ with $\ch(O_{\PP(\mathcal{E})}(1))^{i}:H_{*+i}(\PP(\mathcal{E})) \ra H_{*}(\PP(\mathcal{E}))$. Then the homomorphism
\[
\sum_{i=0}^{n}\xi^{(i)}: \bigoplus_{i=0}^{n} H_{*+i-n}(X) \ra H_{*}(\PP(\mathcal{E}))
\]
is an isomorphism.

($\operatorname{EH}$) Let $E \ra X$ be a vector bundle of rank $r$ over $X$ in $\mathcal{C}$, and let $p: V \ra X$ be an $E$-torsor. Then $p^{*}: H_{*}(X) \ra H_{*+r}(V)$ is an isomorphism. 

($\operatorname{CD}$) For integers $r, n > 0$, let $W=\PP^{n} \times \cdots \times \PP^{n}$ ($r$ factors), and let $p_{i}: W \ra \PP^{n}$ be the $i$-th projection. Let $x_{0},\ldots,x_{n}$ be the standard homogeneous coordinates on $\PP^{n}$, let $a_{1},\ldots,a_{r}$ be non-negative integers, and let $j: E \ra W$ be the subscheme defined by $\prod_{i=1}^{r}p_{i}^{*}(x_{n})^{a_{i}}=0$. Suppose that $E$ is in $\cC$. Then $j_{*}: H_{*}(E) \ra H_{*}(W)$ is injective.

\end{definition}

\subsubsection{} 

If $H_{*}$ is an OBM homology theory, for each projective morphism $f$ the homomorphism $H_{*}(f)$ is denoted $f_{*}$ and called the \emph{push-forward along} $f$. For each l.c.i. morphism $g$ the homomorphism $g^{*}$ is called the \emph{l.c.i. pull-back along} $g$.
For any line bundle $L \ra Y$ the operator $\ch(L):H_{*}(Y)\ra H_{*-1}(Y)$ is called the \emph{first Chern class operator of} $L$. A \emph{morphism} $\vartheta: H_{*} \ra H'_{*}$ \emph{of OBM homology theories} is a natural transformation $\vartheta: H_{*} \ra H'_{*}$ of functors $\mathcal{C}' \ra \Ab$ which moreover commutes with the l.c.i. pull-backs and preserves the exterior products. The axioms give $H_{*}(\Spec k)$ a commutative, graded ring structure, give to each $H_{*}(X)$ the structure of $H_{*}(\Spec k)$-module, and imply that the operations $f_{*}$ and $g^{*}$ preserve the $H_{*}(\Spec k)$-module structure. For each smooth quasiprojective $X \in \Schk$ with structure morphism $f: X \ra \Spec k$, the element $f^{*}1$ is denoted by $1_{X}$.


\subsubsection{}   \label{sub.formalgl}
Recall that a (commutative) formal group law on a commutative ring $R$ is a power series $F_R(u,v)\in R\llbracket u,v\rrbracket $ satisfying
\begin{enumerate}[(a)]
\item $F_R(u,0) = F_R(0,u) = u$,
\item $F_R(u,v)=  F_R(v,u)$,
\item $F_R(F_R(u,v),w) = F_R(u,F_R(v,w))$.
\end{enumerate}
Thus 
\[ F_R(u,v) = u+v +\sum_{i,j>0} a_{i,j} u^i v^j,\]
where $a_{i,j}\in R$ satisfy $a_{i,j}=a_{j,i}$ and some additional relations coming from property (c). 
A formal group law of the form $(F_{+}(u,v)=u+v,R)$ is called additive. Likewise, a formal group law of the form $(F_{\times}(u,v)=u+v-\beta u v,R)$, for some $\beta \in R$, is called multiplicative; and it is in addition called periodic if $\beta \in R$ is invertible. There exists a universal formal group law $(F_\LL,\LL)$, and its coefficient ring $\LL$ is called the \emph{Lazard ring}. In other words, to define a formal group law $(F,R)$ on a commutative ring $R$ is equivalent to give a ring homomorphism $\LL \ra R$. The Lazard ring can be constructed as the quotient of the polynomial ring $\ZZ[A_{i,j}]_{i,j>0}$ by the relations imposed by the three axioms above. 
The images of the variables $A_{i,j}$ in the quotient ring are the coefficients $a_{i,j}$ of the formal group law $F_\LL$. The ring $\LL$ is graded, with $A_{i,j}$ having degree $i+j-1$. 

\subsubsection{} 
Levine and Morel proved that any OBM homology theory $H_{*}$ on an l.c.i-closed admissible subcategory $\cC$ of $\Schk$ has an associated formal group law for its first Chern classes (see \cite[Definition 2.2.1, Corollary 4.1.8, Definition 4.1.9, Proposition 5.2.6]{Levine-Morel}). More precisely, there exists a unique formal group law $(F_{H},H_{*}(\Spec k))$ such that for any smooth quasiprojective scheme $Y$ in $\Schk$ and for any pair of line bundles $L$ and $M$ on $Y$ in $\cC$, one has
\[
\ch(L \otimes M)(1_{Y}) = F_{H}(\ch(L),\ch(M))(1_{Y}).
\]


\subsection{Example: Chow Theory}  \label{ex.Chow} The Chow group functor $A_{*}$ is an OBM homology theory on $\Schk$. The existence of projective push-forwards, l.c.i. pull-backs and exterior products satisfying the required axioms can be found in \cite{FultonIT}. Moreover, given line bundles $L$ and $M$ over the same base one has the formula
\[
\ch(L \otimes M) = \ch(L) + \ch(M).
\] 
Hence, the formal group law of Chow theory is additive given by $(F_{A}(u,v)=u+v,\ZZ)$. In fact, Levine and Morel show in \cite[Theorem 7.1.4]{Levine-Morel} that in characteristic zero $A_{*}$ is universal among OBM homology theories with an additive formal group law on any l.c.i.-closed admissible subcategory $\cC$ of $\Schk$.


\subsection{Example: $K$-Theory}    \label{ex.K.theory}

For any scheme $X$ in $\Schk$, let $G_{0}(X)$ be the Grothendieck $K$-group of coherent $\mathcal{O}_{X}$-modules (see \cite[Chapter 15]{FultonIT}). The class of any coherent sheaf $\mathcal{F}$ on $X$ in the group $G_{0}(X)$ is denoted by $[\mathcal{F}]$. Following Levine and Morel \cite{Levine-Morel}, for each $X$ in $\Schk$ we denote by $G_{0}(X)[\beta,\beta^{-1}]$ the group $G_{0}(X) \otimes_{\ZZ} \ZZ[\beta, \beta^{-1}]$, where $\ZZ[\beta, \beta^{-1}]$ is the ring of Laurent polynomials in a variable $\beta$ of degree $1$. 
One can verify that the functor $G_{0}[\beta, \beta^{-1}]$ becomes an OBM homology theory with the projective push-forwards, l.c.i. pull-backs and exterior products given as follows. 
Define the push-forward for a projective morphism $\fxy$ by
\[
f_{*}([\mathcal{F}] \cdot \beta^{m}):= \sum^{\infty}_{i=0}(-1)^{i}[R^{i}f_{*}(\mathcal{F})] \cdot \beta^{m} \in G_{0}(Y)[\beta, \beta^{-1}]
\]
for any coherent $\mathcal{O}_{X}$-module $\mathcal{F}$ and any $m \in \ZZ$, where $R^{i}f_{*}$ denotes the $i$-th right derived functor of $f_*$. This sum is finite since $f$ is projective. 
The pull-back along an l.c.i. morphism $\gxy$ of relative dimension $d$ is defined by 
\[
g^{*}([\mathcal{F}] \cdot \beta^{m}):= 
\sum^{\infty}_{i=0}(-1)^{i}[L_{i}g^{*}(\mathcal{F})]  \cdot \beta^{m+d} = 
\sum^{\infty}_{i=0}(-1)^{i}  [ \operatorname{Tor}_i^{\mathcal{O}_{Y}}(\mathcal{O}_{X},\mathcal{F})   ]     \cdot \beta^{m+d}
\in G_{0}(X)[\beta, \beta^{-1}]
\]
for any coherent $\mathcal{O}_{Y}$-module $\mathcal{F}$ and any $m \in \ZZ$, where $L_{i}g^{*}$ denotes the $i$-th left derived functor of $g^*$. This sum is finite since $g$ is an l.c.i. morphism. 
For any $X$ and $Y$ in $\Schk$ the external product $$G_{0}(X)[\beta, \beta^{-1}] \times G_{0}(Y)[\beta, \beta^{-1}] \ra G_{0}(X \times Y)[\beta, \beta^{-1}]$$ is given by
\[
([\mathcal{F}_{1}]\cdot \beta^{m_1} ,[\mathcal{F}_{2}] \cdot \beta^{m_2}) \mapsto  [\pi^{*}_{X}(\mathcal{F}_{1}) \otimes_{X \times Y}   \pi^{*}_{Y}(\mathcal{F}_{2})] \cdot \beta^{m_1 + m_2},
\]
for any coherent $\mathcal{O}_{X}$-module $\mathcal{F}_{1}$ and any coherent $\mathcal{O}_{Y}$-module $\mathcal{F}_{2}$, where $\pi_{X}: X \times Y \ra X$ and $\pi_{X}: X \times Y \ra Y$ are the projections. Moreover, given line bundles $L$ and $M$ over the same base one has the formula
\[
\ch(L \otimes M) = \ch(L) + \ch(M) - \beta \cdot \ch(L) \circ \ch(M).
\] 
Hence, the formal group law of $K$-theory is multiplicative and periodic given by $(F_{K}(u,v)= u + v - \beta uv, \ZZ[\beta,\beta^{-1}])$.


\subsection{Example: Algebraic Cobordism}   \label{ex.cobordism}

Algebraic cobordism theory $\Omega_{*}$ was constructed by Levine and Morel in \cite{Levine-Morel}. We now recall the geometric presentation of this theory given by Levine and Pandharipande in \cite{Levine-Pandharipande}.

For $X$ in $\Schk$, let $\cM(X)$ be the set of isomorphism classes of projective morphisms $f: Y\to X$ for $Y \in \Schk$ smooth quasiprojective. This set is a monoid under disjoint union of the domains; let $\Mp(X)$ be its group completion. The class of $f: Y\to X$ in $\Mp(X)$ is denoted by $[f: Y\to X]$.

A double point degeneration is a morphism $\pi: Y\to \PP^1$, with $Y \in \Schk$ smooth quasiprojective of pure dimension, such that $Y_\infty = \pi^{-1}(\infty)$ is a smooth divisor on $Y$ and $Y_0=\pi^{-1}(0)$ is a union $A\cup B$ of smooth divisors intersecting transversely along $D=A\cap B$. Let $N_{A/D}$ and $N_{B/D}$ be the normal bundles of $D$ in $A$ and $B$, and define $\PP(\pi) = \PP(N_{A/D}\oplus \cO_D)$ (notice that $[\PP(N_{A/D}\oplus \cO_D) \ra X]=[\PP(N_{B/D}\oplus \cO_D)\ra X]$ because $N_{A/D} \otimes N_{B/D} \cong \cO_{Y}(A+B)|_D \cong \cO_D $). 

Let $X\in \Schk$, and let $Y \in \Schk$ be smooth quasiprojective of pure dimension. Let $p_1, p_2$ be the two projections of $X\times \PP^1$.  A double point relation is defined by a projective morphism $\pi: Y\to X\times\PP^1$, such that $p_2\circ \pi: Y\to \PP^1$ is a double point degeneration. Let 
\[ [Y_\infty \to X], \quad [A\to X],\quad [B\to X], \quad [\PP(p_2\circ\pi)\to X] \]
be the elements of $\Mp(X)$ obtained by composing with $p_1$. The double point relation is 
\[ [Y_\infty \to X] -[A\to X] - [B\to X] + [\PP(p_2\circ\pi)\to X] \in \Mp(X). \]

Let $\Rel(X)$ be the subgroup of $\Mp(X)$ generated by all the double point relations. The cobordism group of $X$ is defined to be
\[ \Omega_*(X) = \Mp(X)/\Rel(X).\]
The group $\Mp(X)$ is graded so that $[f: Y\to X]$ lies in degree $\dim Y$ when $Y$ has pure dimension. Since double point relations are homogeneous, this grading gives a grading on $\Omega_*(X)$. 
It is shown in \cite{Levine-Morel} that the graded group $\Omega_*(\Spec k)$ is isomorphic to $\LL$.

The existence of projective push-forwards, l.c.i. pull-backs and exterior products satisfying the required axioms can be found in \cite{Levine-Morel}. Moreover, one of the main results in \cite{Levine-Morel} is that in characteristic zero, $\Omega_{*}$ is the universal OBM homology theory on any l.c.i.-closed admissible subcategory $\cC$ of $\Schk$. In other words, for any OBM homology theory $H_{*}$ on $\cC$, there is a unique natural transformation of OBM homology theories $\vartheta: \Omega_{*} \ra H_{*}$.
Given line bundles $L$ and $M$ over the same base one has the formula
\[
\ch(L \otimes M) = F_{\LL}(\ch(L),\ch(M)),
\] 
and therefore the formal group law of $\Omega_{*}$ is the universal formal group law ($F_{\LL}$,$\LL$).


\subsection{Universality}      \label{ex.universality}
We summarize Theorem 1.4.6 from \cite{DaiThesis}. Let ($F_{R}$,$R$) be a formal group law corresponding to the ring homomorphism $\LL \ra R$. We define an OBM homology theory $\Omega^{F_{R}}_{*}$ on $\Schk$ by
\[
\Omega^{F_{R}}_{*} := \Omega_{*} \otimes_{\LL} R.
\]
More precisely, $\Omega^{F_{R}}_{*}(X) := \Omega_{*}(X) \otimes_{\LL} R$ for each $X \in \Schk$, and the projective push-forwards, l.c.i. pull-backs and exterior products are induced by functoriality by those of $\Omega_{*}$. From the universality of $\Omega_{*}$, it follows that in characteristic zero $\Omega^{F_{R}}_{*}$ is universal among OBM homology theories with formal group law $(F_{R},R)$ on any l.c.i. closed admissible subcategory of $\Schk$. 

The formal group laws $(F(u,v)=u+v,\ZZ)$ and $(F(u,v)=u+v-\beta uv,\ZZ[\beta,\beta^{-1}])$ are respectively the universal additive and the universal multiplicative periodic formal group laws. The homomorphisms $\LL \ra \ZZ$ and $\LL \ra \ZZ[\beta,\beta^{-1}]$ classifying these formal group laws allow us to define the OBM homology theories $\Omega^{+}_{*} := \Omega_{*} \otimes_{\LL} \ZZ$ and $\Omega^{\times}_{*} := \Omega_{*} \otimes_{\LL} \ZZ[\beta,\beta^{-1}]$ on any l.c.i. closed admissible subcategory $\cC$ of $\Schk$. It follows from the universality of algebraic cobordism as an OBM homology theory \cite[Theorem 7.1.3]{Levine-Morel} that these are respectively the universal additive OBM homology theory on $\cC$ and the universal multiplicative periodic OBM homology theory on $\cC$. As we mentioned earlier, Levine and Morel showed that the canonical natural transformation $\Omega^{+}_{*} \ra A_{*}$ is an isomorphism \cite[Theorem 7.1.4]{Levine-Morel}. The main goal of this article is to show in Theorem \ref{thm.universality} that the canonical natural transformation $\Omega^{\times}_{*} \ra \Kth$ is an isomorphism.


\section{Universality of $K$-theory}   \label{section.exact.sequences}

In this section we prove the universality of $K$-theory. We start by stating a version of Dai's theorem proved in \cite{DaiThesis} which will constitute the central part of our argument. The assertions in (a) and (b) in the following theorem are merely restatements of Lemma 2.1.4 and Corollary 1.6.4 proved by Dai in \cite{DaiThesis}.

\begin{theorem}[Dai] \label{thm.dai} Let $k$ be a field of characteristic zero. The canonical natural transformation of OBM homology theories $\Omega_{*} \ra \Kth$ descends to a natural transformation of OBM homology theories $\Omega^{\times}_{*} \ra \Kth$ that satisfies:
\begin{enumerate}
\item[(a)] \label{thm.dai.surj} For every $X$ in $\Schk$ the homomorphism $\Omega^{\times}_{*}(X) \ra G_{0}(X)[\beta,\beta^{-1}]$ is surjective. 
\item[(b)] \label{thm.dai.iso} For every quasiprojective scheme $X$  in $\Schk$ the homomorphism $\Omega^{\times}_{*}(X) \ra G_{0}(X)[\beta,\beta^{-1}]$ is an isomorphism.
\end{enumerate}
\end{theorem}


\begin{definition}[Envelopes] \label{envelopes} An \emph{envelope} of a scheme $X$ in $\Schk$ is a proper morphism $\pi: \tilde{X} \ra X$ such that for every subvariety $V$ of $X$ there is a subvariety $\tilde{V}$ of $\tilde{X}$ that is mapped birationally onto $V$ by $\pi$. If $\pi: \tilde{X} \ra X$ is an envelope, we say that it is a projective envelope if $\pi$ is a projective morphism. Likewise, we say that the envelope $\pi: \tilde{X} \ra X$ is birational if for some dense open subset $U$ of $X$, $\pi$ induces an isomorphism $\pi|:\pi^{-1}(U) \ra U$. The composition of envelopes is again an envelope and a morphism obtained by base change from an envelope is again an envelope. 
The domain $\tilde{X}$ of an envelope is not required to be connected, then by Chow's lemma and induction on dimension it follows easily that for every scheme $X$ in $\Schk$ there exists a birational envelope $\pi: \tilde{X} \ra X$ such that $\tilde{X}$ is quasiprojective.
\end{definition}

Gillet in \cite{GilletSeq} established a descent sequence for envelopes in Chow theory and in $K$-theory. We state the $K$-theory version only:

\begin{theorem} [Gillet] \label{thm-Gillet} Let $\pi: \tilde{X}\to X$ be a projective envelope in $\Schk$ and let $p_1, p_2: \tilde{X}\times_X\tilde{X} \to \tilde{X}$ be the two projections. Then the sequence
 \[ G_0(\tilde{X}\times_X\tilde{X}) \xrightarrow{{p_1}_*-{p_2}_*} G_0(\tilde{X}) \stackrel{\pi_*}{\longrightarrow} G_0(X)\to 0\]
is exact.
\end{theorem}

Exactness of Gillet's descent sequence is equivalent to exactness of the sequence:
\[ G_0(\tilde{X}\times_X\tilde{X})[\beta,\beta^{-1}] \xrightarrow{{p_1}_*-{p_2}_*} G_0(\tilde{X})[\beta,\beta^{-1}]  \stackrel{\pi_*}{\longrightarrow} G_0(X)[\beta,\beta^{-1}] \to 0.\]

The analogous descent sequence for algebraic cobordism was proved in \cite{descentseq}:

\begin{theorem} \label{thm-seq1} Let $\pi: \tilde{X}\to X$ be a projective envelope in $\Schk$ and let $p_1, p_2: \tilde{X}\times_X\tilde{X} \to \tilde{X}$ be the two projections. Then the sequence
 \[ \Omega_*(\tilde{X}\times_X\tilde{X}) \xrightarrow{{p_1}_*-{p_2}_*} \Omega_*(\tilde{X}) \stackrel{\pi_*}{\longrightarrow} \Omega_*(X)\to 0\]
is exact.
\end{theorem}

Since tensor product is a right exact functor, we also get an exact sequence
 \[ \Omega^{\times}_*(\tilde{X}\times_X\tilde{X}) \xrightarrow{{p_1}_*-{p_2}_*} \Omega^{\times}_*(\tilde{X}) \stackrel{\pi_*}{\longrightarrow} \Omega^{\times}_*(X)\to 0.\]

We are now ready to prove the universality theorem.

\begin{theorem}[Universality of $K$-theory] \label{thm.universality} Let $k$ be a field of characteristic zero. Let $\cC$ be an l.c.i.-closed admissible subcategory of $\Schk$. Then $\Kth$ is the universal multiplicative periodic oriented Borel-Moore homology theory on $\cC$. 
\end{theorem}
\begin{proof}

Algebraic cobordism $\Omega_{*}$ is the universal OBM homology theory on $\cC$, hence there is a canonical natural transformation of OBM homology theories $\phi: \Omega_{*} \ra \Kth$ on $\cC$. Since the formal group law of $\Kth$ is multiplicative and periodic, then $\phi$ induces a natural transformation of OBM homology theories $\psi: \Omega^{\times}_{*} \ra \Kth$. From the universality of $\Omega_{*}$, we know that $\Omega^{\times}_{*}$ is the universal multiplicative periodic OBM homology theory on $\cC$. Therefore, it is enough to show that $\psi_{X}: \Omega^{\times}_{*}(X) \ra G_{0}(X)[\beta,\beta^{-1}]$ is an isomorphism for each $X$ in $\cC$. 

We choose an envelope $\pi: \tilde{X}\to X$, such that $\tilde{X}$ is quasi-projective. Consider the commutative diagram with exact rows:
\begin{equation} \label{diagram}
\xymatrixcolsep{0.6in}\xymatrix{
\Omega^{\times}_{*}(\tilde{X} \times_{X} \tilde{X}) \ar[d]^{\psi_{\tilde{X} \times_{X} \tilde{X}}} \ar[r]^-{p_{1*}-p_{2*}} & 
\Omega^{\times}_{*}(\tilde{X}) \ar[d]^{\psi_{\tilde{X}}} \ar[r]^{\pi_*} &
\Omega^{\times}_{*}(X) \ar[d]^{\psi_{X}} \ar[r] & 
0 \\
G_{0}(\tilde{X} \times_{X} \tilde{X})[\beta,\beta^{-1}] \ar[r]^-{p_{1*}-p_{2*}} &
G_{0}(\tilde{X})[\beta,\beta^{-1}] \ar[r]^{\pi_*} &
G_{0}(X)[\beta,\beta^{-1}] \ar[r] & 0.
}
\end{equation}
The maps $\psi_{\tilde{X} \times_{X} \tilde{X}}$ and $\psi_{\tilde{X}}$ are isomorphisms by Dai's theorem since both  ${\tilde{X} \times_{X} \tilde{X}}$ and ${\tilde{X}}$ are quasiprojective. The five-lemma implies now that $\psi_X$ is also an isomorphism.
\end{proof}

\begin{remark}
The content of Theorem \ref{thm.universality} is that over a field of characteristic zero, $K$-theory admits a unique natural transformation of OBM homology theories $\vartheta: \Kth \ra H_{*}$ to any OBM homology theory $H_{*}$ with a multiplicative periodic formal group law over any l.c.i.-closed admissible subcategory $\mathcal{C}$ of $\Schk$. In the proof, one establishes the isomorphism of OBM homology theories 
\[
\Omega_{*} \otimes_{\LL} \ZZ[\beta,\beta^{-1}] \cong G_{0}[\beta,\beta^{-1}],
\]
over any such $\mathcal{C}$, via the transformation induced by the universality of algebraic cobordism.
\end{remark}


\appendix

\section{Descent for $K$-theory}   \label{section.universality}

We give here a slightly simplified proof of the descent sequence for $K$-theory Theorem~\ref{thm-Gillet}. The argument is very close to Gillet's original proof. It is shorter because we do not consider simplicial schemes. 
Instead we use the descent sequence in algebraic cobordism (Theorem \ref{thm-seq1}) and Dai's universality (Theorem \ref{thm.dai}).

It suffices to prove the theorem for envelopes $\pi: \tilde{X}\to X$ where $\tilde{X}$ is  quasiprojective. Diagram~\ref{diagram} in the proof of Theorem~\ref{thm.universality} then implies that $\Omega^\times_*(X) \to G_0(X)[\beta,\beta^{-1}]$ is an isomorphism. Since every $X$ admits such an envelope, this result is true for any $X$. Now if $\pi: \tilde{X} \to X$ is an arbitrary projective envelope, then in the Diagram~\ref{diagram} the vertical maps are isomorphisms and the top row is exact, hence the bottom row is also exact.

To prove the theorem for all envelopes of a fixed $X$ where $\tilde{X}$ is quasiprojective, we proceed by Noetherian induction, assuming the result holds for proper closed subschemes $Z$ of $X$ and envelopes $\tilde{Z}\to Z$ with $\tilde{Z}$ quasiprojective. It suffices to consider the case where $\pi: \tilde{X}\to X$ is birational (and $\tilde{X}$ is quasiprojective). As before, this proves the isomorphism  $\Omega^\times_*(X) \to G_0(X)[\beta,\beta^{-1}]$, and then for a not necessarily birational envelope $\pi:\tilde{X}\to X$, with $\tilde{X}$ quasiprojective, Diagram~\ref{diagram} has vertical maps isomorphisms and the top row exact, hence the bottom row is also exact.

Let $\pi: \tilde{X} \ra X$ be a projective birational envelope in $\Schk$, with $\tilde{X}$ quasiprojective, and let $U$ be a nonempty open subscheme of $X$ such that $\pi| : \pi^{-1}(U) \ra U$ is an isomorphism.  Let $Z=X\smallsetminus U$, $\tilde{U}=\pi^{-1}(U)$, $\tilde{Z}=\pi^{-1}(Z)$. Consider the following commutative diagram with maps as indicated:
\begin{equation} \label{diagram.quillen}
\xymatrixcolsep{0.5in}\xymatrix{
&G_{1}(\tilde{U})\ar[d]^{\delta} \ar[r]^-{(\pi|_{\tilde{U}})_*} &G_{1}(U) \ar[d]^{\delta}   \\
G_{0}(\tilde{Z} \times_{Z} \tilde{Z})\ar[d]^{i'_*} \ar[r]^-{\tilde{p}_{1*}-\tilde{p}_{2*}} &G_{0}(\tilde{Z})\ar[d]^{\tilde{i}_{*}} \ar[r]^-{(\pi|_{\tilde{Z}})_*} &G_{0}(Z)\ar[d]^{i_*} \ar[r] & 0 \\
G_{0}(\tilde{X} \times_{X} \tilde{X})\ar[r]^-{p_{1*}-p_{2*}}  &G_{0}(\tilde{X})\ar[r]^-{\pi_*} \ar[d]^{\tilde{j}^*} &G_{0}(X)\ar[d]^{j^*} \ar[r] & 0  \\
&G_{0}(\tilde{U})\ar[r]^-{(\pi|_{\tilde{U}})_*} \ar[d] &G_{0}(U)\ar[d]   \\
&0   &0.
}
\end{equation}
The columns are exact by Quillen's localization theorem \cite{Quillen} (with $G_1$ standing for the first higher $K$-group). The second row (the descent sequence for $Z$) is exact by induction assumption. We need to prove exactness of the third row. The surjectivity of $\pi_*$ holds by the five lemma applied to the columns; and since push-forwards are functorial, the following straightforward diagram chase yields the exactness in the middle.  

Let $\alpha_{0}\in G_0(\tilde{X})$ lie in the kernel of $\pi_*$. Since $(\pi|_{\tilde{U}})_* : G_{0}(\tilde{U})\ra G_{0}(U)$ is an isomorphism, then $\tilde{j}^{*}(\alpha_{0})=0$, and therefore there exist $\alpha_{1} \in G_{0}(\tilde{Z})$ such that $\tilde{i}_{*}(\alpha_{1})=\alpha_{0}$. Let $\alpha_{2} = (\pi|_{\tilde{Z}})_{*}(\alpha_{1})$, and by the vertical exactness choose $\alpha_{3} \in G_{1}(U)$ such that $\delta(\alpha_{3})=\alpha_{2}$. Let $\alpha_{4}=((\pi|_{\tilde{U}})_{*})^{-1}(\alpha_{3}) \in G_1(\tilde{U})$, and let $\alpha_{5}=\delta(\alpha_{4})\in G_{0}(\tilde{Z})$. By construction $\tilde{i}_{*}(\alpha_{1}-\alpha_{5})=\alpha_{0}$ and $(\pi|_{\tilde{Z}})_*(\alpha_{1}-\alpha_{5})=0$. We can choose $\alpha_{6} \in  G_{0}(\tilde{Z} \times_{Z} \tilde{Z})$ such that $(\tilde{p}_{1*}-\tilde{p}_{2*})(\alpha_{6}) = \alpha_{1}-\alpha_{5}$. We let $\alpha_{7}=i'_{*}(\alpha_{6}) \in G_{0}(\tilde{X} \times_{X} \tilde{X})$, and notice that in this way $(p_{1*}-p_{2*})(\alpha_{7}) = \alpha_{0}$. \qed


\bibliographystyle{plain}
\bibliography{cobordismbib}

\end{document}